\documentclass{amsart}

\usepackage{amsmath,amsthm,amssymb}
\usepackage{graphicx}
\usepackage{tikz-cd}
\usepackage{paralist}
\usepackage{environ}
\usepackage{xcolor}
\usepackage{hyperref}
\usepackage{centernot}

\usepackage{marvosym}

\usepackage[framemethod=TikZ]{mdframed}

\usepackage{xcolor,cancel}

\usepackage{booktabs}
\usepackage{tabularx}
\usepackage{multirow,bigstrut}
\usepackage{bigstrut}
\usepackage{CJKutf8}

\makeatletter
\DeclareFontFamily{U}{tipa}{}
\DeclareFontShape{U}{tipa}{bx}{n}{<->tipabx10}{}
\newcommand{\arc@char}{{\usefont{U}{tipa}{bx}{n}\symbol{62}}}%

\newcommand{\arc}[1]{\mathpalette\arc@arc{#1}}

\newcommand{\arc@arc}[2]{%
  \sbox0{$\m@th#1#2$}%
  \vbox{
    \hbox{\resizebox{\wd0}{\height}{\arc@char}}
    \nointerlineskip
    \box0
  }%
}
\makeatother

\makeatletter
\newcommand{\doublewedge}{\big@doubleop{\wedge}}
\newcommand{\big@doubleop}[1]{%
  \DOTSB\mathop{\mathpalette\big@doubleop@aux{#1}}\slimits@
}

\newcommand\big@doubleop@aux[2]{%
  \sbox\z@{$\m@th#1#2$}%
  \makebox[1.35\wd\z@][s]{$\m@th#1#2\hss#2$}%
}

\newcommand{\abs}[1]{\left|#1\right|}     

\newcommand{\Int}{\mbox{int}}
\newcommand{\bdy}{\mbox{bdy}}

\newcommand{\Nrv}{\mbox{Nrv}}
\newcommand{\assign}{\mathrel{\mathop :}=}

\newcommand{\sh}{\mbox{sh}}

\newcommand{\cyc}{\mbox{cyc}}
\newcommand{\hcyc}{\mbox{hCyc}}

\usepackage{pstricks,pst-text,pst-grad,pst-node,pst-3dplot,pstricks-add,pst-poly,pst-coil,pst-bar,pst-all,pst-plot,pst-2dplot,pst-3dplot,pst-solides3d} 
\usepackage{pst-fun,pst-blur}
\usepackage{pst-all}
\usepackage{auto-pst-pdf}

\usepackage{subfigure}
\renewcommand{\thesubfigure}{\thefigure.\arabic{subfigure}}
\makeatletter
\renewcommand{\p@subfigure}{}
\renewcommand{\@thesubfigure}{\thesubfigure:\hskip\subfiglabelskip}
\makeatother

\theoremstyle{plain}
\newtheorem{theorem}{Theorem}
\newtheorem{lemma}{Lemma}
\newtheorem{proposition}{Proposition}
\newtheorem{remark}{Remark}
\newtheorem{definition}{Definition}
\newtheorem{conjecture}{Conjecture}

\newtheorem{example}{Example}

\usepackage[title]{appendix}

\begin{document}

\title[Path Triangulation]{Path Triangulation, Cycles and Good Covers on Planar Cell Complexes.\\ 
Extension of J.H.C. Whitehead's Homotopy System Geometric Realization and E.C. Zeeman's Collapsible Cone Theorems.}

\author[James F. Peters]{James F. Peters}
\address{
Computational Intelligence Laboratory,
University of Manitoba, WPG, MB, R3T 5V6, Canada and
Department of Mathematics, Faculty of Arts and Sciences, Ad\.{i}yaman University, 02040 Ad\.{i}yaman, Turkey
}
\email{james.peters3@umanitoba.ca}
\thanks{The research has been supported by the Natural Sciences \&
Engineering Research Council of Canada (NSERC) discovery grant 185986 
and Instituto Nazionale di Alta Matematica (INdAM) Francesco Severi, Gruppo Nazionale per le Strutture Algebriche, Geometriche e Loro Applicazioni grant 9 920160 000362, n.prot U 2016/000036 and Scientific and Technological Research Council of Turkey (T\"{U}B\.{I}TAK) Scientific Human
Resources Development (BIDEB) under grant no: 2221-1059B211301223.}


\date{}
\dedicatory{Dedicated to J.H.C. Whitehead \& S.V. Banavar}

\begin{abstract}
This paper introduces path triangulation of points in a bounded, simply connected surface region, replacing ordinary triangles in a Delaunay triangulation with path triangles from homotopy theory.  A {\bf path triangle} has a border that is a sequence of paths $h:I\to X, I=[0,1]$. The main results in this paper are that (1) a cone $D\times I$ collapses to a path triangle $h\bigtriangleup K$, extending E.C. Zeeman's dunce hat cone triangle collapse theorem, (2) an ordinary path triangle with geometrically realized straight edges generalizes Veech's billiard triangle, (3) a billiard ball $K\times I$ collapses to a round path triangle geometrically realized as a triangle with curviliear edges, (4) a geometrically realized homotopy system defined in terms of free group presentations of path triangulations of finite cell complexes extends J.H.C. Whitehead's homotopy system geometric realization theorem and (5) every path triangulation of a cell complex is a good cover.
\end{abstract}


\maketitle
\tableofcontents


\section{Introduction}
In this paper, the conventional Delaunay triangulation of a set of points is replaced by a path triangulation of distinguished points $P$ on a bounded, simply connected surface $S$, resulting in a collection of path triangles geometrically realized as curvilinear triangles that provide a good cover of the surface. This form of triangulation results in a collection of path-connected points in overlapping path cycles. 

A {\bf path} $h:[0,1]\to S$ is a continuous mapping from the unit interval to a set of points $P$ in a bounded, simply connected surface $S$.  The surface $S$ is {\bf simply connected}, provided every path $h$ has end points $h(0),h(1)\in P$ and $h$ maps to no self-loops.

Paths either lie entirely on a surface boundary in the planar case or puncture a surface in the non-planar case.  Paths that puncture a surface are called cross-cuts.  A {\bf cross cut path} (also called an {\bf ideal arc}~\cite[\S 3, p.11]{Mosher1988surface}) has both ends in $P$ and path interior in the interior of $S$. A pair of points $p,q$ in $P$ are called {\bf orphan points}, provided there is no path $h$ with endpoints $h(0) = p$ and $h(1) = q$.  Surface cross cuts are useful in collapsing 3D shapes such as cones and spheres to 2D triangles.  A surface {\bf cross section} is a collection of cross-cuts.   

Two forms of path triangles are possible: (1) {\bf surface path triangle} with edges and interior entirely on the surface boundary or (2) {\bf cross-cut path triangle} with edges on the surface boundary and with triangle interior in the interior of the surface.

A path triangulation of a bounded surface $S$ (denoted by $h\bigtriangleup^n S$) is geometrically realized as a collection of 1-cycles on sequences of triangle edges. A {\bf 1-cycle} is a sequence of edges with pairwise shared vertexes with no end vertex, geometrically realized as a simple closed curve. 
Let $p,q\in h\bigtriangleup^n E$ for a selected vertex $p$ (called a generator) and any other vertex $q$. A move operation $+(p,q)\to kp$ results from traversing the $k$ edges from $p$ to reach $q$, leading to a free group presentation of $E$.  \
Free group presentations of path triangulations result in an extension of the homotopy geometric realization theorem introduced by J.H.C. Whitehead~\cite{Whitehead1949BAMS-CWtopology,Whitehead1949BAMS-CWtopology-II}, namely,\\
\vspace{3mm}

{\bf Theorem}~\cite[Theorem 2,\S 6]{Whitehead1949BAMS-CWtopology-II}.
A given homotopy system, $\rho$, has a geometric realization, if dim$\leq 4$.\\
\vspace{3mm}

In its simplest form, a free group presentation of a triangulation $\rho$ is a free group, which is realized geometrically as a collection of path-connected 1-cycles (see Theorem~\ref{thm:pathHsysRealization}).  


Path triangulation leads to the geometric realization of new forms of 2-cells (triangles)~\cite{PetersVergili2021homotopicCycles}, homotopic nerves~\cite{Peters2021KievConf} and good covers~\cite{PetersVergili2021goodCovers} on a planar triangulated cell complex.

Different forms of path triangles result from an extension of the Zeeman's collapsible dunce hat theorem~\cite{Zeeman1963dunceHat}: A dunce hat cone $D\times I$ is collapsible to a triangle (see Theorem~\ref{thm:dunceHat} and Lemma~\ref{lemma:collapses}).  A cell complex $K$ collapses to a subcomplex $L\subset K$, provided there is a finite sequence of elementary collapses starting with $K$ and ending with $L$ (denoted by $K\searrow L$)~\cite[p.342]{Zeeman1963dunceHat}  Here is Zeeman's conjecture (1) (still open):
If $K$ is a contractible 2-complex, then $K\times I$ is collapsible,{\em i.e.}, $K$ collapses to a single vertex $v$ in $K$ via a sequences of elementary collapses $K_i$:
\[
K = K_0\searrow K_1\searrow K_i\searrow\cdots\searrow K_n = v.
\]
For a recent study of Zeeman's conjecture (1), see A. Kupers~\cite[\S 2, p. 37]{Kupers2021Zeeman}.

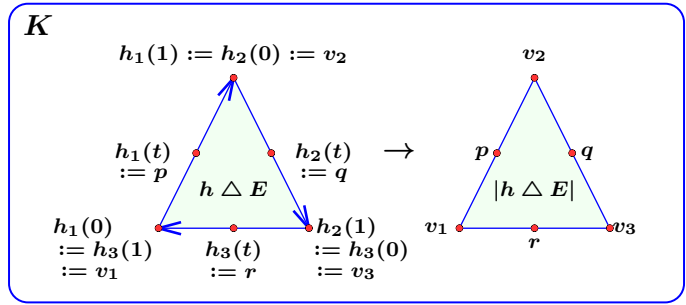
\begin{figure}[!ht]
\begin{pspicture}
(-3.0,-1.0)(6.5,4.5)
\centering
\psframe[linewidth=0.75pt,linearc=0.25,cornersize=absolute,linecolor=blue](-2.0,-1.0)(7.0,3.0)%
\psline*[linecolor=green!5]%
(0.0,0.0)(1.0,2.0)(2.0,0.0)(0.0,0.0)
\psline[linewidth=0.5pt,linecolor=blue,arrowscale=1.0]{-v}%
(0.0,0.0)(1.0,2.0)
\psline[linewidth=0.5pt,linecolor=blue,arrowscale=1.0]{-v}%
(1.0,2.0)(2.0,0.0)
\psline[linewidth=0.5pt,linecolor=blue,arrowscale=1.0]{-v}%
(2.0,0.0)(0.0,0.0)
\psdots[dotstyle=o,dotsize=2.8pt,linewidth=1.2pt,linecolor=black,fillcolor=red!80]%
(0.0,0.0)(0.5,1.0)(1.0,0.0)(1.0,2.0)(1.5,1.0)(2.0,0.0)
\rput(1.0,2.3){\footnotesize $\boldsymbol{h_1(1)\assign h_2(0)\assign v_2}$}
\rput(-0.2,1.0){\footnotesize  $\boldsymbol{h_1(t)}$}
\rput(-0.2,0.7){\footnotesize  $\boldsymbol{\assign p}$}
\rput(2.2,1.0){\footnotesize  $\boldsymbol{h_2(t)}$}
\rput(2.2,0.7){\footnotesize  $\boldsymbol{\assign q}$}
\rput(2.5,0.0){\footnotesize  $\boldsymbol{h_2(1)}$}
\rput(2.7,-0.3){\footnotesize  $\boldsymbol{\assign h_3(0)}$}
\rput(2.5,-0.6){\footnotesize  $\boldsymbol{\assign v_3}$}
\rput(-1.0,0.0){\footnotesize  $\boldsymbol{h_1(0)}$}
\rput(-0.7,-0.3){\footnotesize  $\boldsymbol{\assign h_3(1)}$}
\rput(-0.9,-0.6){\footnotesize  $\boldsymbol{\assign v_1}$}
\rput(1.0,-0.3){\footnotesize  $\boldsymbol{h_3(t)}$}
\rput(1.0,-0.6){\footnotesize  $\boldsymbol{\assign r}$}
\rput(1.0,0.5){\footnotesize  $\boldsymbol{h\bigtriangleup E}$}
\rput(3.2,1.0){\large  $\boldsymbol{\to}$}
\psline*[linecolor=green!5]%
(4.0,0.0)(5.0,2.0)(6.0,0.0)(4.0,0.0)
\psline[linewidth=0.5pt,linecolor=blue]%
(4.0,0.0)(5.0,2.0)
\psline[linewidth=0.5pt,linecolor=blue]%
(5.0,2.0)(6.0,0.0)
\psline[linewidth=0.5pt,linecolor=blue]%
(6.0,0.0)(4.0,0.0)
\psdots[dotstyle=o,dotsize=2.8pt,linewidth=1.2pt,linecolor=black,fillcolor=red!80]%
(4.0,0.0)(4.5,1.0)(5.0,2.0)(5.0,0.0)(5.5,1.0)(6.0,0.0)
\rput(3.7,0.0){\footnotesize  $\boldsymbol{v_1}$}
\rput(4.3,1.0){\footnotesize  $\boldsymbol{p}$}
\rput(5.0,2.3){\footnotesize  $\boldsymbol{v_2}$}
\rput(5.7,1.0){\footnotesize  $\boldsymbol{q}$}
\rput(6.2,0.0){\footnotesize $\boldsymbol{v_3}$}
\rput(5.0,-0.2){\footnotesize  $\boldsymbol{r}$}
\rput(-1.6,2.7){\large  $\boldsymbol{K}$}
\rput(5.0,0.5){\footnotesize  $\boldsymbol{\abs{h\bigtriangleup E}}$}
\end{pspicture}
\caption[]{Three Paths that Construct the Sides of a Path Triangle $\boldsymbol{h\bigtriangleup E}$ in a CW space $\boldsymbol{K}$, leading to its geometric realization $\boldsymbol{\abs{h\bigtriangleup E}}$ as a 2-cell (triangle)}
\label{fig:htriangle}
\end{figure}

\begin{figure}[!ht]
\centering
\subfigure[cell complex $K$]
{\label{fig:bornology}
\begin{pspicture}
(-0.0,-0.5)(3.5,3.0)
\psframe[linewidth=0.75pt,linearc=0.25,cornersize=absolute,linecolor=blue](-0.5,-0.5)(5,3.5)
\psdots[dotstyle=o,dotsize=2.8pt,linewidth=1.2pt,linecolor=black,fillcolor=blue!80]
(0.5,1.5)(1.0,2.0)(1.5,2.2)(3.0,1.8)(4.0,2.2)%
(3.5,1.6)(3.0,1.4)(1.0,1.0)(2.0,0.75)(4.0,1.2)%
\psdots[dotstyle=o,dotsize=2.8pt,linewidth=1.2pt,linecolor=black,fillcolor=red!80]
(1.2,1.5)(2.5,1.5)
\rput(1.4,1.6){\footnotesize $\boldsymbol{p}$}
\rput(2.5,1.7){\footnotesize $\boldsymbol{q}$}
\rput(0.3,1.5){\footnotesize $\boldsymbol{v}$}
\end{pspicture}}
\hfil\qquad\qquad
\subfigure[Path-triangulated cell complex $h\bigtriangleup^n K$]
 {\label{fig:bornology0}
\begin{pspicture}
(-0.0,-0.5)(3.5,3.0)
\psframe[linewidth=0.75pt,linearc=0.25,cornersize=absolute,linecolor=blue](-0.5,-0.5)(5,3.5)
\psdots[dotstyle=o,dotsize=2.8pt,linewidth=1.2pt,linecolor=black,fillcolor=blue!80]
(0.5,1.5)(1.0,2.0)(1.5,2.2)(3.0,1.8)(4.0,2.2)%
(3.5,1.6)(3.0,1.4)(1.0,1.0)(2.0,0.75)(4.0,1.2)%
\psline*[linewidth=0.5pt,linecolor=green!20]
(0.5,1.5)(1.0,2.0)(1.5,2.2)(2.5,1.5)(2.0,0.75)(1.0,1.0)(0.5,1.5)
\psline*[linewidth=0.5pt,linecolor=blue!20]
(3.0,1.4)(3.0,1.8)(3.5,1.6)(3.0,1.4)
\psline*[linewidth=0.5pt,linecolor=blue!20]
(3.0,1.4)(3.5,1.6)(4.0,1.2)(3.0,1.4)
\psline*[linewidth=0.5pt,linecolor=blue!20]
(3.0,1.8)(4.0,2.2)(3.5,1.6)(3.0,1.8)
\rput(1.4,1.6){\footnotesize $\boldsymbol{p}$}
\rput(2.5,1.7){\footnotesize $\boldsymbol{q}$}
\rput(0.3,1.5){\footnotesize $\boldsymbol{v}$}
\rput(1.8,1.4){\footnotesize $\boldsymbol{h}$}
\rput(2.4,2.4){\tiny $\boldsymbol{nucleus}$}
\psline[border=0.2pt]{*->}%
(2.4,2.2)(1.45,1.7)
\psline[linewidth=0.5pt,linecolor=blue,arrowscale=0.5]{-v}%
(0.5,1.5)(1.0,2.0)
\psline[linewidth=0.5pt,linecolor=blue,arrowscale=0.5]{-v}%
(1.2,1.5)(0.5,1.5)
\psline[linewidth=0.5pt,linecolor=blue,arrowscale=0.5]{-v}%
(1.0,2.0)(1.5,2.2)
\psline[linewidth=0.5pt,linecolor=blue,arrowscale=0.5]{-v}%
(1.2,1.5)(1.0,2.0)
\psline[linewidth=0.5pt,linecolor=blue,arrowscale=0.5]{-v}%
(1.2,1.5)(2.0,0.75)
\psline[linewidth=0.5pt,linecolor=blue,arrowscale=0.5]{-v}%
(1.5,2.2)(1.2,1.5)
\psline[linewidth=0.5pt,linecolor=blue,arrowscale=0.5]{-v}%
(2.0,0.75)(2.5,1.5)
\psline[linewidth=0.5pt,linecolor=blue,arrowscale=0.5]{-v}%
(1.5,2.2)(3.0,1.8)
\psline[linewidth=0.5pt,linecolor=blue,arrowscale=0.5]{-v}%
(3.0,1.8)(2.5,1.5)
\psline[linewidth=0.5pt,linecolor=blue,arrowscale=0.5]{-v}%
(2.5,1.5)(3.0,1.4)
\psline[linewidth=0.5pt,linecolor=blue,arrowscale=0.5]{v-v}%
(3.0,1.4)(3.0,1.8)
\psline[linewidth=0.5pt,linecolor=blue,arrowscale=0.5]{-v}%
(3.0,1.4)(3.5,1.6)
\psline[linewidth=0.5pt,linecolor=blue,arrowscale=0.5]{-v}%
(3.5,1.6)(3.0,1.8)
\psline[linewidth=0.5pt,linecolor=blue,arrowscale=0.5]{-v}%
(2.5,1.5)(1.5,2.2)
\psline[linewidth=0.5pt,linecolor=blue,arrowscale=0.5]{v-v}%
(1.2,1.5)(2.5,1.5)
\psline[linewidth=0.5pt,linecolor=blue,arrowscale=0.5]{-v}%
(3.0,1.8)(4.0,2.2)
\psline[linewidth=0.5pt,linecolor=blue,arrowscale=0.5]{-v}%
(3.0,1.8)(4.0,2.2)
\psline[linewidth=0.5pt,linecolor=blue,arrowscale=0.5]{-v}%
(4.0,2.2)(3.5,1.6)
\psline[linewidth=0.5pt,linecolor=blue,arrowscale=0.5]{-v}%
(3.5,1.6)(4.0,1.2)
\psline[linewidth=0.5pt,linecolor=blue,arrowscale=0.5]{-v}%
(4.0,1.2)(3.0,1.4)
\psline[linewidth=0.5pt,linecolor=blue,arrowscale=0.5]{-v}%
(3.0,1.4)(2.0,0.75)
\psline[linewidth=0.5pt,linecolor=blue,arrowscale=0.5]{-v}%
(2.0,0.75)(1.0,1.0)
\psline[linewidth=0.5pt,linecolor=blue,arrowscale=0.5]{-v}%
(0.5,1.5)(1.0,1.0)
\psline[linewidth=0.5pt,linecolor=blue,arrowscale=0.5]{-v}%
(1.0,1.0)(1.2,1.5)
\psdots[dotstyle=o,dotsize=2.8pt,linewidth=1.2pt,linecolor=black,fillcolor=red!80]
(1.2,1.5)
\psdots[dotstyle=o,dotsize=2.8pt,linewidth=1.2pt,linecolor=black,fillcolor=blue!80]
(2.5,1.5)
\end{pspicture}}
\hfil
\caption[]{Path-based triangulated Rotman-Vigolo Hawaiian fish cell complex with nucleus vertex $p\in\Int(h\bigtriangleup^n K)$}
\label{fig:HawaiianFish}
\end{figure}

\section{Preliminaries}
We assume that a bounded, simply connected surface $S$ is Hausdorf, {\em i.e.}, surface points $P$ in $S$ reside in disjoint neighborhoods.
A path triangulation $K$ of surface points (denoted by 
$h{\bigtriangleup}^n K$) maps to a collection of triangles that provide an alternative to Delaunay triangulation~
\cite{Delone1932geometry,Delone1934geometry}.  
The basic building block in $h{\bigtriangleup}^n K$ is a path triangle denoted by $h{\bigtriangleup} E$.  

\begin{definition}\label{def:path} {\rm \bf Path}.\\
Let $I = [0,1]$, the unit interval.  A {\bf path} in a space $(S,P)$ is a continuous map $h:I\to S$ with endpoints $h(0)=x_0, h(1)=x_1\in P$ and $h(t)\in S$ for $t\in I$ ~\cite[\S 2.1,p.11]{Switzer2002CWcomplex}.
The geometric realization of a path $h$ is denoted $\abs{h}$.
\qquad \textcolor{blue}{\Squaresteel}
\end{definition}



\begin{lemma}\label{lemma:path}
Every path has a geometric realization as an edge.
\end{lemma} 
\begin{proof}
From Def.~\ref{def:path}, a path $h:I\to X$ includes all points $h(t), t\in [0,1]$, {\em i.e.}, we have the closed set of points
\[
\left\{h(0),\dots,h(t),\dots,h(1)\right\},\ \mbox{for}\ 0\leq t\leq 1,
\]

	
which is geometrically realized as an edge $\abs{h} = \arc{h(0)h(1)}$ with $h(t)\in \Int(\abs{h}), t\in (0,1)$ (interior of edge $\abs{h}$).
\end{proof}

From Lemma~\ref{lemma:path}, the geometric realization of a path $h$ is a 1-cell (edge) $\abs{h}$ in a cell complex. In attaching paths together to construct a path triangle $E$ (see Def.~\ref{def:pathTriangle}), the end result is a path cycle with a geometric counterpart $\abs{E}$ that is a 1-cycle.

\begin{definition}\label{def:filledCycle} {\rm \bf 1-Cycle}.
In a CW space $K$~\cite{Whitehead1949BAMS-CWtopology}, a 1-cycle $E$ (denoted by $\cyc E$) in a CW space $K$ is a collection of path-connected vertexes (0-cells) on edges (1-cells) attached to each other with no end vertex.
\qquad\textcolor{blue}{\Squaresteel}
\end{definition} 
 
\begin{definition}\label{def:pathConn} {\bf Path-Connected}.\\
A pair of 0-cells $v,v'$ in a cell complex $K$ is {\bf path-connected}, provided there is a sequence of paths $h_1,\dots,h_k$, starting with $h_1(0)\assign v$ and ending with $h_{k-1}(1) = h_k(0)\assign v'$.
\qquad \textcolor{blue}{\Squaresteel}
\end{definition}

In defining a path triangle, we consider both the bounding edge as well as the interior of the triangle.
Consider the contour (bounding edge) of a planar cell complex shape $\sh E$ (denoted by $\bdy(\sh E)$) in a CW complex. 


\begin{definition}\label{def:pathTriangle}{\bf Path Triangle}.\\
A {\bf path triangle} $h\bigtriangleup abc$ with a boundary $
\bdy(h\bigtriangleup abc)$ is a sequence of three overlapping paths with no end path in a space $X$.
Each path $h\in h\bigtriangleup abc$ maps to an edge $\ell\in X$. 


Each edge $\ell$ is an edge perpendicular to base $\arc{bc}$, with an endpoint on either $\arc{ab}$ or on $\arc{ac}$  and $\ell\cap \bdy(h\bigtriangleup abc)\neq \emptyset$.  Let the interior $\Int (h\bigtriangleup abc)$ be defined by the homotopy $f:X\times I\to Y$, {\em i.e.},
\begin{align*}
X &= \left\{\ell_i\in h\bigtriangleup abc: \ell_i\perp \arc{bc}\ \mbox{and}\ \ell_i\cap\bdy(h\bigtriangleup E)\neq \emptyset\right\}.\\
Y &= h\bigtriangleup abc.\\
f(\ell_t) &= \ell_t\cap \bdy(h\bigtriangleup abc), t\in I.\\
          &=\mbox{set of boundary points on $h\bigtriangleup abc$ cut by $\ell_t$}.\\
\Int(h\bigtriangleup abc) &= h\bigtriangleup abc\setminus f(\ell_t),\ \mbox{for all}\ t\in I.\\
h\bigtriangleup abc &= \bdy(h\bigtriangleup abc)\cup \Int(h\bigtriangleup abc).
\end{align*}
In the plane, the geometric realization of $h\bigtriangleup E$ is a 2-cell (filled triangle) denoted by $\abs{h\bigtriangleup E}$.
\qquad \textcolor{blue}{\Squaresteel}
\end{definition}

\begin{remark}
For example, a path triangle with an empty interior will result from collapsing a hollow cone\footnote{Many thanks to Tane Vergili for pointing this out} such as the one in Fig.~\ref{fig:duncehatTriangle}. 
\qquad \textcolor{blue}{\Squaresteel}
\end{remark}

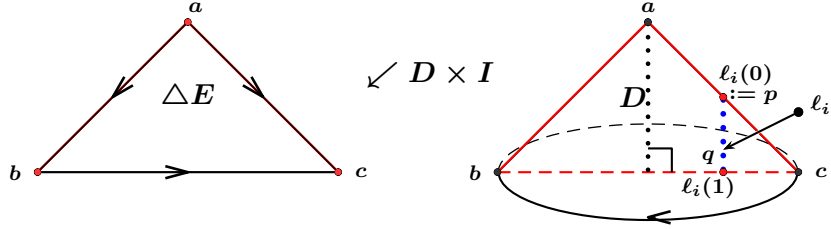
\begin{figure}[!ht]
\begin{pspicture}
(-1.0,0.0)(5.0,4.0)
\centering
\psline[linestyle=solid]%
		(0,1)(2,3)(4,1)(0,1)
\psline[linecap=0,linecolor=red](0,1)(2,3)(4,1)
\psline[linewidth=0.5pt,linecolor=black,arrowscale=1.0]{-v}%
(2.0,3.0)(1.0,2.0)
\psline[linewidth=0.5pt,linecolor=black]%
(1.0,2.0)(0.0,1.0)
\psline[linewidth=0.5pt,linecolor=black,arrowscale=1.0]{-v}%
(2.0,3.0)(3.0,2.0)
\psline[linewidth=0.5pt,linecolor=black]%
(3.0,2.0)(4.0,1.0)
\psline[linewidth=0.5pt,linecolor=black,arrowscale=1.0]{-v}%
(0.0,1.0)(2.0,1.0)
\psline[linewidth=0.5pt,linecolor=black]%
(2.0,1.0)(4.0,1.0)
\psdots[dotstyle=o,dotsize=2.8pt,linewidth=1.2pt,linecolor=black,fillcolor=red!80]%
(0,1)(2,3)(4,1)
\rput(-0.3,1.0){\footnotesize $\boldsymbol{b}$}
\rput(2.1,3.2){\footnotesize $\boldsymbol{a}$}
\rput(4.3,1.0){\footnotesize $\boldsymbol{c}$}
\rput(2.0,2.0){\large $\boldsymbol{\bigtriangleup E}$}
\end{pspicture}
\begin{pspicture}
(-1.0,0.0)(5.0,4.0)
    \psline[linewidth=0.5pt,linecolor=black,arrowscale=1.0]{-v}%
(2.1,0.4)(2.0,0.4)
    \psline[linestyle=solid]%
		(0,1)(2,3)(4,1)
		\psellipticarc[linewidth=0.5pt,linestyle=dashed](2,1)(2,.65){0}{180}
    \psellipticarcn(2,1)(2,.65){0}{180}
    \psline[linecap=0,linecolor=red](0,1)(2,3)(4,1)
		\psline[linestyle=dashed,linecolor=red](0,1)(4,1)
		\psline[linestyle=dotted,linecolor=black,linewidth=1.8pt]%
		(2,1)(2,3)
\psline[linestyle=dotted,dotsize=2.2pt,linewidth=2.2pt,linecolor=blue]%
(3,2)(3,1)
\psdots[dotstyle=o,dotsize=2.9pt,linewidth=1.2pt,linecolor=black,fillcolor=red!90]%
(3,2)(3,1)
\rput(3.4,2.0){\footnotesize $\boldsymbol{\assign p}$}
\rput(2.8,1.2){\footnotesize $\boldsymbol{q}$}
\rput(2.8,0.8){\footnotesize $\boldsymbol{\ell_i(1)}$}
\psline[border=0.2pt]{*->}%
(4.0,1.8)(3.0,1.3)
\rput(4.3,1.9){\footnotesize $\boldsymbol{\ell_i}$}
\rput(3.3,2.3){\footnotesize $\boldsymbol{\ell_i(0)}$}
\psdots[dotstyle=o,dotsize=2.8pt,linewidth=1.2pt,linecolor=black,fillcolor=black!80]%
(0,1)(2,3)(4,1)
\rput(2,1){\psline(0,9pt)(9pt,9pt)(9pt,0)}
\rput(2.0,3.2){\footnotesize $\boldsymbol{a}$}
\rput(-0.3,1.0){\footnotesize $\boldsymbol{b}$}
\rput(4.3,1.0){\footnotesize $\boldsymbol{c}$}
\rput(1.8,2.0){\large $\boldsymbol{D}$}
\rput(-0.9,2.3){\large $\boldsymbol{\swarrow D\times I}$}
\end{pspicture}
\caption[]{Zeeman Dunce Hat (cone) $\boldsymbol{D\times I}$ collapses to a 2-cell $\boldsymbol{\bigtriangleup E}$ with path $\boldsymbol{\ell_i:I\to D}$ piercing $\boldsymbol{D}$ at $\boldsymbol{\ell_i(0)\assign p}$ and $\boldsymbol{\ell_i(1)\assign q}$.}
\label{fig:duncehatTriangle}
\end{figure}

\begin{figure}[!ht]
\begin{pspicture}
(-1.0,0.0)(5.0,4.0)
\centering
\psline*[linestyle=solid,linecolor=green!5]%
		(0,1)(2,3)(4,1)(0,1)
\psline[linecap=0,linecolor=red](0,1)(2,3)(4,1)
\psline[linewidth=0.5pt,linecolor=black,arrowscale=1.0]{-v}%
(0.0,1.0)(2.0,3.0)
\psline[linewidth=0.5pt,linecolor=black,arrowscale=1.0]{-v}%
(2.0,3.0)(4.0,1.0)
\psline[linewidth=0.5pt,linecolor=black,arrowscale=1.0]{-v}%
(4.0,1.0)(0.0,1.0)
\psdots[dotstyle=o,dotsize=2.8pt,linewidth=1.2pt,linecolor=black,fillcolor=red!80]%
(0,1)(2,3)(4,1)
\rput(0.6,2.0){\footnotesize $\boldsymbol{h_1}$}
\rput(-0.5,1.0){\footnotesize $\boldsymbol{h_1(0)}$}
\rput(-0.3,0.7){\footnotesize $\boldsymbol{= h_3(1)}$} %
\rput(-0.5,0.4){\footnotesize $\boldsymbol{\assign v_1}$}
\rput(2.1,3.8){\footnotesize $\boldsymbol{h_1(1)}$}
\rput(2.2,3.5){\footnotesize $\boldsymbol{=h_2(0)}$}
\rput(2.1,3.2){\footnotesize $\boldsymbol{\assign v_2}$}
\rput(3.3,2.0){\footnotesize $\boldsymbol{h_2}$}
\rput(4.7,1.0){\footnotesize $\boldsymbol{h_3(0)}$}
\rput(4.7,0.7){\footnotesize $\boldsymbol{= h_2(1)}$} %
\rput(4.5,0.4){\footnotesize $\boldsymbol{\assign v_3}$} 
\rput(2.0,0.8){\footnotesize $\boldsymbol{h_3}$}
\rput(2.0,2.0){\large $\boldsymbol{h\bigtriangleup E}$}
\end{pspicture}
\begin{pspicture}
(-1.0,0.0)(5.0,4.0)
    \psline*[linestyle=solid,linecolor=green!5]%
		(0,1)(2,3)(4,1)(0,1)
		\psline[linewidth=0.5pt,linecolor=black,arrowscale=1.0]{-v}%
(2.1,0.4)(2.0,0.4)
		\psellipticarc[linestyle=dashed](2,1)(2,.65){0}{180}
    \psellipticarcn(2,1)(2,.65){0}{180}
		\psellipticarcn(2,1)(2,.65){0}{180}
    \psline[linecap=0,linecolor=red](0,1)(2,3)(4,1)
		\psline[linestyle=dashed,linecolor=red](0,1)(4,1)
		\psline[linestyle=dotted,linecolor=black,linewidth=1.2pt]%
		(2,1)(2,3)
\psdots[dotstyle=o,dotsize=2.8pt,linewidth=1.2pt,linecolor=black,fillcolor=black!80]%
(0,1)(2,3)(4,1)
\rput(2,1){\psline(0,9pt)(9pt,9pt)(9pt,0)}
\rput(-0.3,1.0){\footnotesize $\boldsymbol{v_1}$} 
\rput(2.0,3.2){\footnotesize $\boldsymbol{v_2}$} 
\rput(4.3,1.0){\footnotesize $\boldsymbol{v_3}$} 
\rput(1.8,2.0){\large $\boldsymbol{D}$}
\rput(1.8,2.0){\large $\boldsymbol{D}$}
\rput(-0.9,2.3){\large $\boldsymbol{\swarrow D\times I}$}
\end{pspicture}
\caption[]{Zeeman Dunce Hat (filled cone) cell complex $\boldsymbol{D\times I}$ collapses to a filled Path Triangle $\boldsymbol{h\bigtriangleup E}$ over the unit interval $\boldsymbol{I=[0,1]}$, {\em i.e.}, $\boldsymbol{D\times I \searrow h\bigtriangleup E}$, {\em i.e.}, $\boldsymbol{\mbox{cone}\ D\times I}$ collapses via a homotopy to a path 2-cell (triangle) $\boldsymbol{h\bigtriangleup E}$.}
\label{fig:duncePathTriangle}
\end{figure}
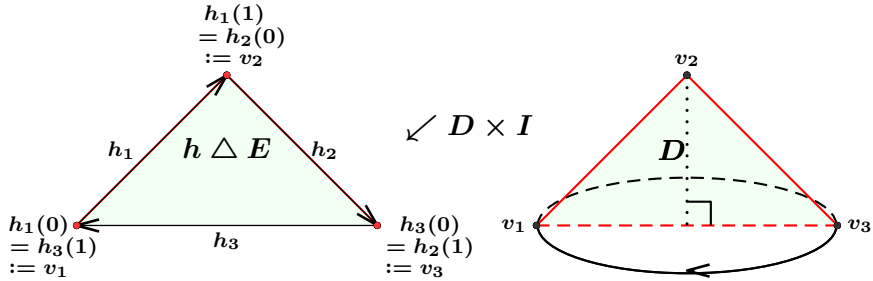

Let $D$ be a cone (a topological dunce hat~\cite{Zeeman1963dunceHat}). 

\begin{theorem}\label{thm:dunceHat}{\rm ~\cite[p.342]{Zeeman1963dunceHat}}.\\
$D\times I$ is collapsible.
\end{theorem}

E.C. Zeeman noted that a dunce hat~\cite{Zeeman1963dunceHat} (cone) $D$ is the simplest polyhedron contractible via a homotopy $f:D\times I\to\ \mbox{cell complex}\ Y$, but not collapsible in the sense of J.H.C. Whitehead~\cite[\S 3]{Whitehead1939homotopy}.

Zeeman's dunce hat has a simply connected surface.
For Riemann~\cite[p. 7]{Riemann1851thesis}, a surface is {\bf simply connected}, provided each cross-cut is a line that pierces the interior of the surface.  Here, a line is a path $\ell:I\to 2^{\mathbb{R}^2}$ with $\ell(t)$ piercing the interior of cone $D$ for $t\in (0,1)$. 

\begin{example}
The collapse of a cone $D\times I$ to a triangle $\boldsymbol{\bigtriangleup E}$ is shown in Fig.~\ref{fig:duncehatTriangle}.  This collapse results from the union of the paths $\ell_i:I\to D$ that cut through the interior of the triangle from points $p$ on the edge of the cone to a points on a line segment $\arc{bc}$ that spans the space between the edges at the base of the cone, {\em i.e.},
\begin{align*}
D\times I &= \left\{\ell_0\right\},\\
\left\{\ell_0\right\} & \cup D\setminus\left\{\ell_0\right\}\times I = \left\{\ell_0,\ell_1\right\},\cdots,\\
\left\{\ell_0,\ell_1,\dots,\ell_i,\dots,\ell_{n-2[n]}\right\}& \cup
D\setminus\left\{\ell_0,\ell_1,\dots,\ell_i,\dots,,\ell_{n-1[n]}\right\}\times I\\ 
         &= \bigcup \ell_i\\
				 &= \bigtriangleup E\in D.
\end{align*}
In other words, the surface of cone $D$ is simply connected so that each line in the deformation of $D$ to a triangle is a cross-cut path, which punctures the interior of the cone between a surface boundary point $p$ and a point $q$ on a line $\arc{bc}$ spanning the base of the cone.  
\qquad \textcolor{blue}{\Squaresteel}
\end{example}


\begin{lemma}\label{lemma:collapses}
Let $\bigtriangleup abc$ be a triangle with vertex $a$ at the top of a cone $D$ with a simply connected surface.  And let base $\arc{bc}$ on $\bigtriangleup abc$ be an edge on the base of $D$.  Then $D\times I\searrow \bigtriangleup abc$. 
\end{lemma}
\begin{proof}
From Theorem~\ref{thm:dunceHat}, $D\times I$ collapses.    Let $\ell:I\to 2^{\mathbb{R}^2}$ be a path that puntures $D$ with $\ell(0)$ on the boundary of $D$ and $\ell(1)$ on $\arc{bc}$ on the base of $D$.  Then let the homotopy
$f:D\times I\to \bigtriangleup abc$ be defined by
\begin{align*}
f(D,i) &= \ell_i\cap \bigtriangleup abc = \left\{p_i\right\}\in \bigtriangleup abc,\ \mbox{with }\ \ell_i\bot\arc{bc}\in D.\\
\bigtriangleup abc &= \mathop{\bigcup}\limits_{0<i\leq \abs{\arc{bc}}} \ell_i.
\end{align*}
Hence, $D\times I$ collapses to $\bigtriangleup abc$.

\end{proof}

\begin{proposition}\label{prop:hTriangle}
Let $h\bigtriangleup v_1v_2v_3$  be a path triangle with vertex $v_2$ at the top of a filled cone $D$ with a bounded, simply connected surface.  And let base $\arc{v_1v_3}$ on $h\bigtriangleup v_1v_2v_3$ be on the base of $D$.  Then $D\times I\searrow   h\bigtriangleup v_1v_2v_3$.
\end{proposition}
 
\begin{proof}
To construct the interior of $h\bigtriangleup v_1v_2v_3$, let $\ell_i:I\to Y$ be a path that punctures the interior of solid cone $\bigtriangleup v_1v_2v_3$.  Then let
\begin{align*}
\ell_i(t) &= \ell_i(t)\cap \Int (\bigtriangleup v_1v_2v_3), t\in \left(0,1\right),\ \ell_i\bot\arc{v_1v_3}\in \bigtriangleup v_1v_2v_3.\\
\Int (h\bigtriangleup v_1v_2v_3) &= \mathop{\bigcup}\limits_{0<i<\abs{\arc{v_1v_3}}} \ell_i(t).
\end{align*}
To construct the edges which are paths of $h\bigtriangleup abc$, let paths $h_l,h_r,h_b:I\to D$ be defined as follows:
\begin{align*}
h_l(0) &\assign v_1,\\
h_l(1) &\assign v_2,\\
h_l(t) &= h_l(t)\cap \Int (\arc{v_1v_2}), t\in \left(0,1\right)\\
h_r(0) &\assign v_2,\\
h_r(1) &\assign v_3,\\
h_r(t) &\assign h_r(0)\cap \Int (\arc{v_2v_3}), t\in \left(0,1\right)\\
h_b(0) &\assign v_1,\\
h_b(1) &\assign v_3,\\
h_b(t) &\assign h_b(0)\cap \Int (\arc{v_1v_3}), t\in \left(0,1\right).
\end{align*}
Then the proof is symmetric with the proof of Lemma~\ref{lemma:collapses} and the desired result follows.
\end{proof}

\begin{remark}
A path triangle $h\bigtriangleup E$ differs from a triangle in a Zeeman dunce hat cone, since paths construct the bounding edges on $h\bigtriangleup E$. 
\qquad \textcolor{blue}{\Squaresteel}
\end{remark}

\begin{example}\label{ex:dunceHatTriangle}
From Lemma~\ref{lemma:collapses}, a filled cone collapses to a path triangle.  A sample $\bigtriangleup v_1v_2v_3\in D$ collapsing of a cone $D$ via a homotopy $D\times I$ to a path triangle $h\bigtriangleup v_1v_2v_3\in D$ (briefly, $h\bigtriangleup E$) is shown in Fig.~\ref{fig:duncePathTriangle}.  Here, a sequence of three paths $h_1,h_2,h_3$ form the edges of the path triangle $h\bigtriangleup E$ with
\begin{align*}
h_1(0) &= h_3(1)\assign v_1.\\
h_1(1) &= h_2(0)\assign v_2.\\
h_3(0) &= h_2(1)\assign v_3.\ \mbox{\qquad \textcolor{blue}{\Squaresteel}}\\
\end{align*}
\end{example}

The path triangle $h\bigtriangleup E$ in Example~\ref{ex:dunceHatTriangle} has a geometric realization $\abs{h\bigtriangleup E}$, which is a filled Euclidean triangle with straight edges such as the ones in a Delaunay trianglulation of a space $X$, {\em i.e.},
\begin{align*}
I &= [0,1]\\
h_1:I&\to X, h_2:I \to X, h_3:I \to X.\\
h_1\cap h_3 &= h_1(0)\assign h_3(1),\\
h_1\cap h_2 &= h_1(1)\assign h_2(0),\\
h_2\cap h_3 &= h_2(1)\assign h_3(0),\\
h_3\cap h_1 &= h_3(1)\assign h_1(0),\\ 
\bdy(h\bigtriangleup E) &= \left\{h_1,h_2,h_3\right\}\ \mbox{(triangle boundary)}.\\
\Int(h\bigtriangleup E) &\neq \emptyset\ \mbox{(nonvoid interior)}.\\
\bigtriangleup E &= \bdy(h\bigtriangleup E)\cup \Int(h\bigtriangleup E).\\ 
h\bigtriangleup E &\mapsto \abs{h\bigtriangleup E}.
\end{align*}

\begin{example}\label{ex:htriangle}
A sample path triangle $h\bigtriangleup h_1(0)h_2(0)h_3(0)$ and its geometric realization $\abs{h\bigtriangleup v_1v_2v_3}$ are shown in Fig.~\ref{fig:htriangle}.  From Lemma~\ref{lemma:path}, we know that every path in $h\bigtriangleup h_1(0)h_2(0)h_3(0)$ has a geometric realization as a 1-cell (edge).
 Here, triangle $\abs{\bigtriangleup v_1v_2v_3}$ is a geometric realization path triangle $h\bigtriangleup h_1(0)h_2(0)h_3(0)$, containing a sequence of overlapping paths $h_1,h_2,h_3$, {\em i.e.},
\begin{align*}
h_1,h_2,h_3 &\in h\bigtriangleup E \mapsto\ \mbox{\bf edge}\ \arc{v_1v_2}\in \abs{\bigtriangleup v_1v_2v_3},\ \mbox{with}\\ 
h_1(0)&\assign h_3(1)\assign v_1,h_1(1)\assign h_2(0)\assign v_2\ \mbox{and}\\
h_1(t)&\assign\ \mbox{\bf vertex}\  p\in \mbox{edge}\ \arc{v_1v_2}, t\in [0,1].\\
h_1,h_2,h_3 &\in h\bigtriangleup E \mapsto\ \mbox{\bf edge}\ \arc{v_2v_3}\in \abs{\bigtriangleup v_1v_2v_3},\ \mbox{with}\\ 
h_1(1)&\assign h_2(0)\assign v_2,h_2(1)\assign h_3(0)\assign v_3 \ \mbox{and}\\
h_2(t)&\assign\ \mbox{\bf vertex}\ q\in \mbox{edge}\ \arc{v_2v_3}, t\in [0,1].\\
h_1,h_2,h_3 &\in h\bigtriangleup E\mapsto\ \mbox{\bf edge}\ \arc{v_3v_1}\in \abs{\bigtriangleup v_1v_2v_3},\ \mbox{with}\\ 
h_2(1)&\assign h_3(0)\assign v_3,h_3(1)\assign h_1(0)\assign v_1 \ \mbox{and}\\
h_3(t)&\assign\ \mbox{\bf vertex}\ r\in \mbox{edge}\ \arc{v_3v_1}, t\in [0,1].\ \mbox{\qquad \textcolor{blue}{\Squaresteel}}
\end{align*}
\end{example}

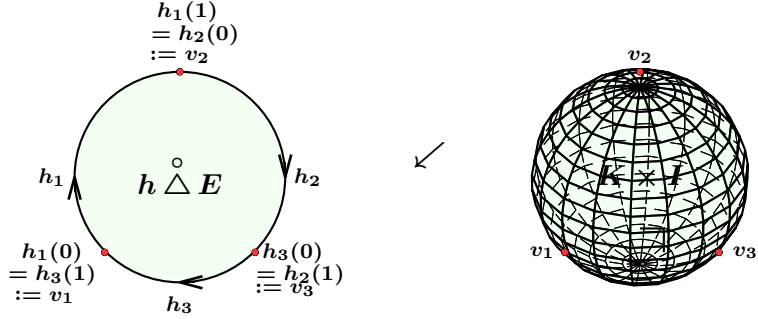
\begin{figure}[!ht]
\begin{pspicture} 
(-1.0,0.0)(5.0,5.0)
\centering
\pscircle*[fillstyle=solid,fillcolor=green!5]%
(2,2){1.4}
\pscircle[fillstyle=solid,fillcolor=green!5]%
(2,2){1.4}
\psline[linewidth=0.5pt,linecolor=black,arrowscale=1.0]{-v}%
(0.60,1.8)(0.60,2.0)
\psline[linewidth=0.5pt,linecolor=black,arrowscale=1.0]{-v}%
(3.4,2.2)(3.4,2.0)
\psline[linewidth=0.5pt,linecolor=black,arrowscale=1.0]{v-}%
(2.0,0.6)(2.2,0.6)
\psdots[dotstyle=o,dotsize=2.8pt,linewidth=1.2pt,linecolor=black,fillcolor=red!80]%
(1,1)(2,3.4)(3,1)
\rput(0.30,2.0){\footnotesize $\boldsymbol{h_1}$}
\rput(0.3,1.0){\footnotesize $\boldsymbol{h_1(0)}$}
\rput(0.3,0.7){\footnotesize $\boldsymbol{= h_3(1)}$} %
\rput(0.2,0.4){\footnotesize $\boldsymbol{\assign v_1}$}
\rput(2.1,4.2){\footnotesize $\boldsymbol{h_1(1)}$}
\rput(2.2,3.9){\footnotesize $\boldsymbol{=h_2(0)}$}
\rput(2.0,3.6){\footnotesize $\boldsymbol{\assign v_2}$}
\rput(3.7,2.0){\footnotesize $\boldsymbol{h_2}$}
\rput(3.5,1.0){\footnotesize $\boldsymbol{h_3(0)}$}
\rput(3.6,0.7){\footnotesize $\boldsymbol{= h_2(1)}$} %
\rput(3.4,0.5){\footnotesize $\boldsymbol{\assign v_3}$} 
\rput(2.0,0.3){\footnotesize $\boldsymbol{h_3}$}
\rput(2.0,2.0){\large $\boldsymbol{h\mathop{\bigtriangleup}\limits^{\circ} E}$}
\end{pspicture}
\begin{pspicture}
(-1.0,0.0)(5.0,4.0)
\rput(2,2){
    \psSolid[
        object=sphere,
        r=0.5,
        linecolor=black,
        fillcolor=green!5,
        action=draw*,mode=4,
        ngrid=18 18
    ]}
\psdots[dotstyle=o,dotsize=2.8pt,linewidth=1.2pt,linecolor=black,fillcolor=red!80]%
(1.0,1)(2.0,3.4)(3.05,1.0)
\rput(2,1){\psline(0,9pt)(9pt,9pt)(9pt,0)}
\rput(0.7,1.0){\footnotesize $\boldsymbol{v_1}$} 
\rput(2.0,3.6){\footnotesize $\boldsymbol{v_2}$} 
\rput(3.4,1.0){\footnotesize $\boldsymbol{v_3}$} 
\rput(2.0,2.0){\large $\boldsymbol{K\times I}$}
\rput(-0.8,2.3){\large $\boldsymbol{\swarrow}$}
\end{pspicture}
\caption[]{Billiard ball cell complex $\boldsymbol{K\times I}$ Zeeman-style collapses to round Path Triangle $\boldsymbol{h\mathop{\bigtriangleup}\limits^{\circ} E}$, {\em i.e.}, $\boldsymbol{\mbox{sphere}\ K\times I}$ collapses via a homotopy to a path 2-cell (triangle) $\boldsymbol{h\mathop{\bigtriangleup}\limits^{\circ} E}$.}
\label{fig:roundPathTriangle}
\end{figure}

Let $K$ be a filled sphere 
 in a space $\mathbb{R}^3$. A {\bf round triangle} $abc$ (denoted by $\mathop{\bigtriangleup}\limits^{\circ} abc\in K$) is a sequence of edges on the circumference of a circle that lies on a slice of a solid sphere with edges on the surface of the sphere.   

\begin{remark}
Paths in a path triangle resemble trajectories of a particle in spacetime.  From a Physics perspective, a path triangle would be viewed as a sequence of trajectories with overlapping endpoints. In keeping with W.A. Veech's introduction of triangular billiards~\cite{Veech1989triangularBilliards}, spheres in this work are called billiard balls. Unlike Veech's billiard triangles with straight edges, the homotopy collapse of a billiard ball results in a round triangle. Round path triangles are useful in covering a cell complex with a curvilinear boundary.  For example, in the path triangles $h\bigtriangleup E$ in the triangulation $h\bigtriangleup^n K$ of a cell complex $K$ with curved edges such as those typically found in video frame foreground shapes, we have
\begin{align*}
K &= \mathop{\bigcup}\limits_{h\bigtriangleup E\in h\bigtriangleup^n K}h\bigtriangleup E.\\
\mbox{nerve}\ \Nrv K &= \mathop{\bigcap}\limits_{h\bigtriangleup E\in h\bigtriangleup^n K}h\bigtriangleup E \neq \emptyset.
\end{align*} 
and $\Nrv K$ contracts to a vertex $p$ in the intersection of the path triangles in $h\bigtriangleup^n K$.  In other words, $h\bigtriangleup^n K$ provides a good cover of cell complex $K$.
\end{remark}

\begin{proposition}\label{prop:collapseSpher2Circle}
Let $\mathop{\bigtriangleup}\limits^{\circ} abc$ be a round path triangle in a billiard ball complex $K$ in a space $X$ with triangle edges $ab,ab,bc$ on the boundary of $K$.  And let base $bc$ on $h\bigtriangleup abc$ be an edge on the base of $D$.  
Sphere $sph K\times I\searrow h\bigtriangleup abc$ collapses to a round triangle $\mathop{\bigtriangleup}\limits^{\circ} abc$.
\end{proposition}
\begin{proof}
Replace cone $D$ with sphere {$sph K$} 
 and triangle $\bigtriangleup abc$ with round triangle $\mathop{\bigtriangleup}\limits^{\circ} E$ in the proof of Prop~\ref{prop:hTriangle} and the desired result follows.
\end{proof} 

\begin{proposition}
A sphere $sph K\times I\searrow h\mathop{\bigtriangleup}\limits^{\circ} abc$ collapses to a round path triangle $h\mathop{\bigtriangleup}\limits^{\circ} E$.
\end{proposition}
\begin{proof}
Replace triangle $\bigtriangleup abc$ with round path triangle $h\mathop{\bigtriangleup}\limits^{\circ} E$ in the proof of Prop.~\ref{prop:collapseSpher2Circle} and the desired result follows from Lemma~\ref{lemma:collapses}.
\end{proof} 

\begin{example}
A sample collapse of a billiard ball complex to a round triangle is shown in Fig.~\ref{fig:roundPathTriangle}.
\qquad \textcolor{blue}{\Squaresteel}
\end{example}

\begin{theorem}\label{prop:1cycle}
Every path triangle has a geometric realization as a 1-cycle.
\end{theorem}
\begin{proof}
From Def.~\ref{def:pathTriangle}, a path triangle $h\bigtriangleup K$ is a collection of path-connected vertexes with no end vertex and with nonvoid interior. From Lemma~\ref{lemma:path},
the geometric realization of each path in $h\bigtriangleup K$ is an edge in its geometric realization $\abs{\bigtriangleup K}$. Consequently, the sequence of paths in $h\bigtriangleup K$ results in a sequence of 1-cells (edges) attached to each other. Hence, from Def.~\ref{def:filledCycle}, $\abs{\bigtriangleup K}$ is a 1-cycle.
\end{proof}

A {\bf path-based triangulation} of a cell complex $K$ (denoted by $h\bigtriangleup^n K$) is a collection of path triangles in which adjacent path triangles either have a common vertex or a common edge.  The geometric realization of $h\bigtriangleup^n K$ is denoted by $\abs{h\bigtriangleup^n K}$.

\begin{example}
A sample bounded cell complex $K$ is shown in Fig.~\ref{fig:bornology}. The path-based triangulation of cell complex $h\bigtriangleup^n K$ displayed in Fig.~\ref{fig:bornology0} constucts an example of a Rotman-Vigolo Hawaiian fish, inspired by the bird complex in~\cite[Fig. 11.2, p. 367]{Rotman1965TheoryOfGroups} and Hawaiian earring in~\cite[\S 3.31, p. 79]{Vigolo2018HawaiianEarrings}. 
\qquad \textcolor{blue}{\Squaresteel}
\end{example}

\begin{definition}\label{def:orphan} {\bf Orphan Vertex}.\\
Let $K$ be a finite, bounded cell complex.  An {\bf orphan
vertex} (briefly, {\bf orphan}) is a vertex is not an end vertex of a path and is located either on the boundary $\bdy E$ or on an edge in the interior $\Int K$ of $K$.
\qquad \textcolor{blue}{\Squaresteel}
\end{definition} 

\begin{example}
Vertexes $p,q$ in Fig.~\ref{fig:bornology} are examples of orphans that cease being orphans in Fig.~\ref{fig:bornology0},
since path $h$ constructs an edge attached $p$ and $q$.
\qquad \textcolor{blue}{\Squaresteel}
\end{example}

%


A path-based triangulation of a cell complex $K$ results in the elimination of orphan vertexes in $K$, {\em i.e.},

\begin{definition}\label{def:triangulation} {\bf Path Triangulation}.\\
A {\bf path triangulation} $h\bigtriangleup^n K$ of a cell complex $K$ results from the introduction of a path between every pair of orphans in the complex.
\qquad \textcolor{blue}{\Squaresteel}
\end{definition}


Orphans are distinguished surface points that eventually cease being orphans as a result of a path triangulation.

\begin{lemma}\label{lemma:pathBasedTriangulation}
Every path triangulation of a cell complex containing 3 or more ophan 0-cells results in a collection of path triangles.
\end{lemma}
\begin{proof}
Immediate from Def.~\ref{def:triangulation}.
\end{proof}

\begin{theorem}\label{thm:pathTriangleRealization}{\bf Path Triangulation Geometric Realization}.\\
Every path triangulation of a cell complex has a geometric realization as a collection of 1-cycles.
\end{theorem}
\begin{proof}
Let $h\bigtriangleup K$ be a path triangulation of a cell complex $K$. From Lemma~\ref{lemma:pathBasedTriangulation}, $h\bigtriangleup E$ is a collection of path triangles.  Hence, from Theorem~\ref{prop:1cycle}, every path triangle in $h\bigtriangleup K$ has a geometric realization as a 1-cycle, which is the desired result. 
\end{proof}

\begin{lemma}\label{lemma:pathConnected}
Every pair of vertexes in a path triangulation is path-connected.
\end{lemma}
\begin{proof}
Let $h\bigtriangleup^n K$ be a path-based triangulation of a cell complex $K$ and let $v,v'$ be vertexes in $K$.  From Def.~\ref{def:triangulation}, there are no orphans in $h\bigtriangleup^n K$. Hence, there is a path $h\in h\bigtriangleup^n K$ with $h(0)\assign v$ and $h(1)$ is the beginning of one or more paths $h'$.  If $h(1)\assign v'$, then $h$ is desired path.  Otherwise, we follow path $h'$ until we reach vertex $v'$.  As a result, every pair of vertexes in $h\bigtriangleup^n K$ is path-connected.
\end{proof}

Recall that a {\bf triangulation} $\bigtriangleup^n K$ of a collection of points (sites) in a bounded planar region $K$ is Delaunay (aka Delone) if and only if the circumcircle of none of its triangles contains sites in its interior~\cite{Thrivikraman2012Delaunay}.  Here, path triangulation eliminates orphan points.  Hence, there are no orphans in the interior any path triangle. 

\begin{proposition}\label{prop:triangleInterior}
There are no orphans in the interior of the triangles in a path triangulation.
\end{proposition}
\begin{proof}
This is an immediate consequence of Def.~\ref{def:triangulation}.
\end{proof} 

\begin{remark}
The proposed path triangulation alternative to Delaunay triangulation has a number attractive features, namely,
\begin{compactenum}[1$^o$]
\item For its definition, a path triangle $h\bigtriangleup E$ does not depend on a circumcircle with the absence of sites (orphan vertexes) in its interior, since, from Def.~\ref{def:triangulation}, all orphans are endpoints of paths on edges of path in a path triangulation and, from Prop.~\ref{prop:triangleInterior}, path triangles do not have orphans in their interiors.  
\item Every path triangulation has a free group presentation (see Theorem~\ref{theorem:pathTriangulationGroup}).  Such a free group has simplicity and reflects the path structure of the triangulation.  In a simple extension of Whitehead's homotopy system realization theorem, every free group presentation of a path triangulation has a geometric realization as a collection of triangles (unlike Euclidean triangles, the edges of the triangles covering a complex can be either straight or curved (see, {\em e.g.}, $\bigtriangleup ABC$ in the proof of Prop. 1 in~\cite{EuclidElements300BC} vs. $\mathop{\bigtriangleup}\limits^{\circ} E$ in Fig.~\ref{fig:roundPathTriangle}).
\item Unlike a Delaunay triangle, every path triangle has either a nonvoid or an empty interior. \\
\item Unlike a Zeeman dunce hat triangle or Veech billiard triangle, every point in a path triangle $h\bigtriangleup E$ edge is a known point that is either an endpoint or an interior point in a path that constructs a triangle edge in the path triangulation of distinguished surface orphan points, since we know $h(t)$ for every $t\in I$ in $h:I\to K$.  
\item Every path-based triangulation of a planar cell complex is a good cover of the complex (see Theorem~\ref{thm:pathTriangleGoodCover}). \qquad \textcolor{blue}{\Squaresteel}
\end{compactenum}
\end{remark}

\begin{figure}[!ht]
	\centering
	\begin{pspicture}
	(-4.5,-1.0)(6.5,4.5)
	\psframe[linewidth=0.75pt,linearc=0.25,cornersize=absolute,linecolor=blue](-4.0,-1.0)(6.5,4.4)
	\pscustom[linewidth=0.5pt,linecolor=blue]{%
		\moveto(-3.0,0)
		\curveto(-2,5)(-0.5,2)(-0.0,3.8)
		\moveto(-3.0,0)
		\curveto(-2,5.5)(-0.5,1.5)(-0.0,3.8)
		\moveto(-3.0,0)
		\curveto(-2,6.5)(-0.5,0.5)(-0.0,3.8)
		\moveto(-3.0,0)
		\curveto(-2,8)(-0.5,-1)(-0.0,3.8)
	}
		\pscustom[linewidth=0.5pt,linecolor=blue]{%
		\moveto(-0.0,3.8)
		\curveto(2,2)(0.5,1)(0,1.0)
		\moveto(-0.0,3.8)
		\curveto(2,2.5)(0.5,2.5)(0,1.0)
		\moveto(-0.0,3.8)
		\curveto(2,3)(0.5,3)(0,1.0)
		\moveto(-0.0,3.8)
		\curveto(2,4)(0.5,3.5)(0,1.0)
	}
	\pscustom[linewidth=0.5pt,linecolor=blue]{%
		\moveto(-3.0,0)
		\curveto(-2.0,2)(-1.5,0)(0,1.0)
		\moveto(-3.0,0)
		\curveto(-2.0,2.5)(-1.5,-0.5)(0,1.0)
		\moveto(-3.0,0)
		\curveto(-2.0,3)(-1.5,-1)(0,1.0)
		\moveto(-3.0,0)
		\curveto(-2.0,4)(-1.5,-2)(0,1.0)
	}
	\psellipse[linecolor=red,linestyle=dotted,linewidth=1.2pt]%
	(1.0,2.5)(0.90,0.25)
	\psellipse[linecolor=red,linestyle=dotted,linewidth=1.2pt]%
	(-2.5,1.8)(0.80,0.25)
	\psellipse[linecolor=red,linestyle=dotted,linewidth=1.2pt]%
	(-0.3,0.5)(0.80,0.55)
\psdots[dotstyle=o,dotsize=2.5pt,linewidth=1.5pt,linecolor=black,fillcolor=red!80]
	(-3.0,0)(-0.0,3.8)(0.0,1)
	\rput(-3.5,4.1){ $\boldsymbol{K}$}
	\rput(-2.0,-0.2){\footnotesize {\color{black} $\boldsymbol{h_i(0)=\ell_i(1)\assign v_1}$}}
	\rput(0,4.0){\footnotesize {\color{black} $\boldsymbol{h_i(1)=k_i(0)\assign v_2}$}}
	\rput(1.0,1.0){\footnotesize {\color{black} $\boldsymbol{k_i(1)=}$}}
	\rput(1.0,0.7){\footnotesize {\color{black} $\boldsymbol{\ell_i(0)\assign}$}}
	\rput(1.0,0.4){\footnotesize {\color{black} $\boldsymbol{v_3}$}}
	\rput(0.25,2.5){\footnotesize $\boldsymbol{[k]}$}
	\rput(-2.9,1.8){\footnotesize $\boldsymbol{[h]}$}
	\rput(0.0,0.5){\footnotesize $\boldsymbol{[\ell]}$}
	\rput(1.5,2.0){\large  $\boldsymbol{\mapsto}$}
	\pscustom[linewidth=0.5pt,linecolor=black]{%
		\moveto(2.0,0)
		\curveto(2,5)(4.5,2)(5.0,3.8)
		\moveto(2.0,0)
		\curveto(2,5.5)(4.5,1.5)(5.0,3.8)
		\moveto(2.0,0)
		\curveto(2,6.5)(4.5,0.5)(5.0,3.8)
		\moveto(2.0,0)
		\curveto(2,8)(4.5,-1)(5.0,3.8)
	}
		\pscustom[linewidth=0.5pt,linecolor=black]{%
		\moveto(5.0,3.8)
		\curveto(7,2)(5.5,1)(5,1.0)
		\moveto(5.0,3.8)
		\curveto(7,2.5)(5.5,2.5)(5,1.0)
		\moveto(5.0,3.8)
		\curveto(7,3)(5.5,3)(5,1.0)
		\moveto(5.0,3.8)
		\curveto(7,4)(5.5,3.5)(5,1.0)
	}
	\pscustom[linewidth=0.5pt,linecolor=black]{%
		\moveto(2.0,0)
		\curveto(3.0,2)(3.5,0)(5,1.0)
		\moveto(2.0,0)
		\curveto(3.0,2.5)(3.5,-0.5)(5,1.0)
		\moveto(2.0,0)
		\curveto(3.0,3)(3.5,-1)(5,1.0)
		\moveto(2.0,0)
		\curveto(3.0,4)(3.5,-2)(5,1.0)
	}
\psdots[dotstyle=o,dotsize=2.5pt,linewidth=1.5pt,linecolor=black,fillcolor=red!80]
(2.0,0)(5.0,3.8)(5.0,1)
\rput(-1.0,2.0){{\color{black} $\boldsymbol{[h]\bigtriangleup E}$}}
\rput(4.0,2.0){{\color{black} $\boldsymbol{\abs{[h]\bigtriangleup E}}$}}
\rput(2.0,-0.2){\footnotesize {\color{black} $\boldsymbol{v_1}$}}
\rput(5.0,4.0){\footnotesize {\color{black} $\boldsymbol{v_2}$}}
\rput(5.2,0.8){\footnotesize {\color{black} $\boldsymbol{v_3}$}}
	\end{pspicture}
\caption[]{$\boldsymbol{[h]\bigtriangleup E}$ (Curviliner path class triangle) $\mapsto$ $\boldsymbol{\abs{[h]\bigtriangleup E}}$ (geometrically realized as curviliner, overlapping triangles)}
\label{fig:pathClassTriangles}
\end{figure}
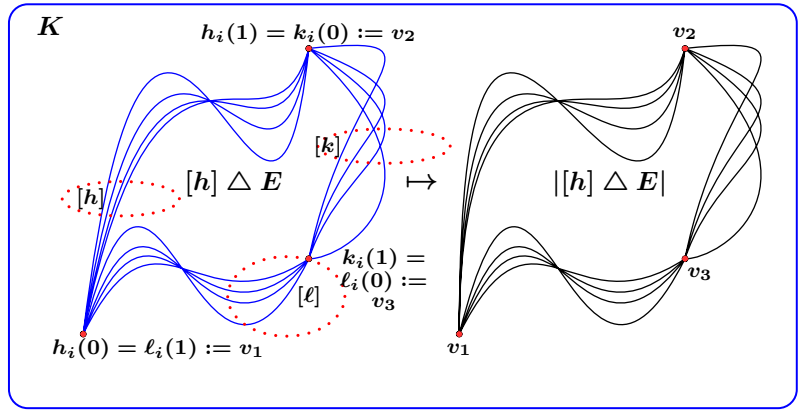

\section{Path Class Triangles}
This section introduces path class triangles.

\begin{definition}\label{def:pathClass} {\bf Path-Class}.\\
A {\bf path-class} $h$ (denoted by $\boldsymbol{[h]}$) is a collection of paths $h$ that have the same initial and final values, i.e., if path $h_i,h_j\in [h]$, then
$h_i(0) = h_j(0)$ and $h_i(1) = h_j(1)$.
\qquad \textcolor{blue}{\Squaresteel}
\end{definition}

\begin{definition}\label{def:pathClassTriangle} {\bf Path-Class Triangle}.\\
A {\bf path-class triangle} $\boldsymbol{[h]\bigtriangleup E}$ is a collection of three path classes that pairwise have the same initial and final values.
\qquad \textcolor{blue}{\Squaresteel}
\end{definition}

\begin{example}\label{ex:pathClassTriangle} {\bf Sample Path-Class Triangle}.\\
A sample {\bf path-class triangle} $\boldsymbol{[h]\bigtriangleup E = \left\{[h],[k],[\ell]\right\}}$ is shown in Fig.~\ref{fig:pathClassTriangles}, {\em i.e.}, we have
\begin{align*}
h_i & \in [h], k_i\in [k], \ell_i\in [\ell], i\in \left\{1,2,3,4\right\}.\\
h_i(0) &= \ell_i(1)\assign\ \mbox{0-cell}\ v_1,\\
h_i(1) &= k_i(0)\assign\ \mbox{0-cell}\ v_2,\\
k_i(1) &= \ell_i(0)\assign\ \mbox{0-cell}\ v_3.
\end{align*}
\qquad \textcolor{blue}{\Squaresteel}
\end{example}

\begin{definition}\label{def:filledNervComplex} {\rm \bf Nerve Complex}.\\
A nerve complex $\Nrv E$ in a CW space is a collection of nonempty cell complexes with nonvoid intersection.  This is an example of an Alexandrov nerve~\cite[\S 4.3,p. 39]{Alexandroff1932elementaryConcepts}.
\textcolor{blue}{\Squaresteel}
\end{definition}

A vertex with a collection of triangles attached to it, is called the {\bf nucleus} of the attached triangles (also called the nucleus of an Alexandrov-Hopf nerve complex or {\bf Nerv}~\cite{AlexandroffHopf1935Topologie}).

\begin{lemma}\label{lemma:AlexandrovNerv}
Every vertex in the triangulation of the vertices in a CW space is the nucleus of an Alexandrov-Hopf nerve complex in which one more triangles are attached to each vertex. 
\end{lemma}

\begin{theorem}\label{thm:Nerv}
A path triangulation of a cell complex containing $n$ vertexes contains $n$ Alexandrov-Hopf nerve complexes. 
\end{theorem}
\begin{proof}
Immediate from Lemma~\ref{lemma:AlexandrovNerv}.
\end{proof}

\begin{definition}\label{def:pathClassTriangulation} {\bf Path Class Triangulation}.\\
A {\bf path class triangulation} $[h]\bigtriangleup^n E$ of a cell complex $K$ results from the introduction of a path class between one or more pairs of neighboring orphans in the complex.
\end{definition}

\begin{proposition}\label{prop:NervComplex}
Every path class triangulation of a cell complex is a collection of Alexandrov-Hopf nerve complexes.
\end{proposition}
\begin{proof}
Let $h\bigtriangleup^n K$ be a path class triangulation of a complex $K$ and let $[h]\bigtriangleup E$ be a path class triangle in $h\bigtriangleup^n K$.  From Lemma~\ref{lemma:AlexandrovNerv}, every vertex in $K$ is the nucleus of an Alexandrov-Hopf nerve complex.  From Def.~\ref{def:pathClassTriangle}, every pair of path classes in $[h]\bigtriangleup E$ with a common orphan vertex have nonvoid intersection equaling the nucleus of an Alexandrov-Hopf nerve complex.  Hence, the desired result follows.
\end{proof}

\begin{proposition}\label{prop:hybridPathTriangulation}
Every path class triangulation of a cell complex is a path triangulation.
\end{proposition}
\begin{proof}
Since every path class contains one or more paths between neighboring orphans as endpoints, the desired result follows.
\end{proof}

\section{Path Cycles}
This section briefly introduces path cycles that have geometric realization as 1-cycles.

\begin{definition}\label{def:filledHomotopicCycle} {\rm \bf  Path Cycle~\cite[App. A]{PetersVergili2021homotopicCycles}}.\\
	A {\bf path cycle} $\hcyc E$ in a CW space $K$ is a sequence of continuous maps $\left\{h_i\right\}_{i=0}^{(n-1)[n]}$, $h_i:I\to K,  [n] =\ \mbox{mod}\ n$, with 
	\begin{align*}
	h_0(0) &= h_{n[n]}(0) = v_0\in K, h_{n-1[n]}(1)= v_{n-1[n]}\in K.\\
	\mbox{Given}\ h_i &\in \left\{h_i\right\}_{i=0}^{(n-1)[n]},\\ 
	g &= h_i(0),\ \mbox{a basis element in $\hcyc E$,  so that}\\
	v &\overbrace{\assign kg =kh_i(0)=h_{i+k}(0)\in K, k\in \mathbb{Z}.}^{\mbox{\textcolor{blue}{\bf $h_{i+k}(0)$\ \mbox{walks forward}\ $k$ vertexes\ \mbox{from $g$ to reach}\ $h_{i+k}(0)$}}}\\
	\bar{h}_{i+k}(0)&\overbrace{\assign h_{(i+k)-k}(0) = h_i(0).}^{\mbox{\textcolor{blue}{\bf $\bar{h}_{i+k}(0)$\ \mbox{walks back}\ $k$ vertexes from $h_{i+k}(0)$ to reach $h_i(0)$}}}\\
	+:K\times K&\to K,\ \mbox{is defined by}\\
	+(v,v') &= \overbrace{+(kg,k'g) = kg+k'g.}^{\mbox{\textcolor{blue}{\bf move to $\boldsymbol{v''}$ via $\boldsymbol{h_i(0)=v,h_{i+k} (v')}$}}}\\
	&= kh_i(0)+k'h_{i+k}(0) = h_{(i+k+k')[n]}(0) = v''\in K.\mbox{\qquad \textcolor{blue}{\Squaresteel}}
	\end{align*}
\end{definition}

%

\begin{theorem}\label{thm:pathTriangulationCycles}
Every path triangulation is collection of path cycles.
\end{theorem}
\begin{proof}
Let $h\bigtriangleup^n K$ (briefly, $K$) be a path triangulation of a cell complex $K$.  From Def.~\ref{def:triangulation}, $K$ is a collection of path triangles $h\bigtriangleup E\in K$. In a triangulation containing more than one path triangle, every vertex $v\in h\bigtriangleup E$ is the beginning of a another path $h'$, {\em i.e.}, $h(1) = h'(0)\assign v$ in a triangle $h'\bigtriangleup E'$.  From Lemma~\ref{lemma:pathConnected}, every pair of vertexes in $K$ is path-connected.
By following the path that begins with vertex $v$, we eventually arrive at a path that takes us back to $v$.  Hence, every vertex in the triangulation belongs to a path cycle and, from Def.~\ref{def:filledHomotopicCycle}, $K$ is a collection of path cycles.
\end{proof}

\section{Free Group Presentation of a Path Triangulation}
This section introduces free group presentation of a path triangulation of a cell complex.

\begin{definition}\label{RotmanPresentation}{{\bf Rotman Presentation}\rm \cite[p.239]{Rotman1965TheoryOfGroups}}$\mbox{}$\\
Let $X = \left\{g_1,\dots\right\},\bigtriangleup = \left\{v = \sum kg_i,v\in \mbox{group} G,g_i\in X\right\}$ be a set of generators of members of a nonempty set $X$ and set of relations between members of $G$ and the generators in $X$.  A mapping of the form $\left\{X,\bigtriangleup\right\}\to G$, a free group, is called a presentation of $G$. \qquad\textcolor{blue}{\Squaresteel}
\end{definition}

We write $G(V,+)$ to denote a group $G$ on a nonvoid set $V$ with a binary operation $+$.  For a group $G(V,+)$ presentable as a collection of linear combinations of members of a basis set $\mathcal{B}\subseteq V$, we write $G(\mathcal{B},+)$.

\begin{definition}\label{def:cellComplexFreeGroup}{\bf Free Group Presentation of a Cell Complex}.$\mbox{}$\\
Let $2^K$ be the collection of cell complexes in a CW space $K$, $E\in 2^K$ containing $n$ vertexes, $G(E,+)$ a group on nonvoid set $E$ with binary operation $+$, $\bigtriangleup = \left\{v = \sum kg_i,v\in E,g_i\in E\right\}$ be a set of generators of members in $E$, set of relations between members of $E$ and the generators $\mathcal{B}\subset E$, $g_i\in\mathcal{B},v = h_{i[n]}(0)\in K$, $k_i$ the $i^{th}$ integer coefficient $[n]$ in a linear combination $\mathop{\sum}\limits_{i,j}k_ig_j$ of generating elements $g_j = h_j(0)\in \mathcal{B}$.   A free group {\bf presentation} of $G$ is a continuous map $f:2^K\times\bigtriangleup\to 2^K$ defined by
\begin{align*}
f(\mathcal{B},\bigtriangleup) &= \left\{v \assign{\mathop{\sum}\limits_{i,j}k_ig_j\in\bigtriangleup}: v\in E,g_j\in \mathcal{B}, k_i\in\mathbb{Z}\right\}\\
          &= \overbrace{\boldsymbol{G(\left\{ g_1,\dots,g_{_{\abs{\tiny \mathcal{B}}}}\right\}},+).}^{\mbox{\textcolor{blue}{\bf $\mathcal{B}\times\bigtriangleup\mapsto$ free group $\boldsymbol{G}$ presentation of $G(E,+)$}}}\mbox{\qquad\textcolor{blue}{\Squaresteel}}
\end{align*}
\end{definition}

\begin{lemma}\label{lemma:1cycleFreeGroup}~\cite{PetersVergili2021homotopicCycles}
Every 1-cycle in a CW space $K$ has a free group presentation.
\end{lemma}

\begin{theorem}\label{theorem:pathTriangulationGroup}
Every path triangulation of a cell complex has a free group presentation.
\end{theorem}
\begin{proof}
Let $h\bigtriangleup^n K$ be a path triangulation of a cell complex $K$.  From Theorem~\ref{thm:pathTriangulationCycles},
$h\bigtriangleup^n K$ is a collection of path cycles. From Lemma~\ref{lemma:1cycleFreeGroup}, every path cycle in $h\bigtriangleup^n K$.  Let $v$ a vertex in a path cycle $\cyc E\in h\bigtriangleup^n K$.  From Lemma~\ref{lemma:pathConnected}, every pair of vertexes $v,v'$ in $h\bigtriangleup^n K$ is path-connected.  Hence, $v$ serves as a generator in a basis for a free group presentation of the triangulation.
\end{proof}

\begin{example}
From Theorem~\ref{theorem:pathTriangulationGroup}, vertex $p$ is a generator of a free group presentation of the path triangulation of the Rotman-Vigolo fish in Fig.~\ref{fig:HawaiianFish}, namely, $G(\mathcal{B},+) = G(\left\{p\right\},+)$.  Vertex $p$ is labeled {\bf nucleus}, since $p$ serves as a generator of not only the free group for the entire fish but also for the free group presentation of the cluster of triangles attached to $p$.  This $p$-cell complex is also calle a maximal nucleus complex (MNC), this complex has the highest number of triangle attached to a vertex in the path triangulation of the fish. 
\qquad\textcolor{blue}{\Squaresteel}
\end{example}

In an extension of Whitehead's homotopy system geometric realization theorem, observe that a {\bf homotopy group} $\rho$ that presents a path-connected cell complex $K$ (denoted by $\rho(K)$ is a free group (a Rotman homotopy group~\cite[370]{Rotman1965TheoryOfGroups}that presents $K$, which has a basis--one or more generator vertexes that are initial path vertexes--and a move binary operation). A {\bf path homotopy group} is a free group that presents a path triangulation of a cell complex $K$ (denoted by $h\rho(K)$).  

\begin{definition}\label{def:pathHomotopySystem}{\bf Path Homotopy System}.\\
A {\bf path homotopy system} $\left\{h\rho(K)\right\}$ $E$ (denoted by $h\mathop{\bigtriangleup}\limits^{\rho(K)} E$) is a collection of path homotopy groups that present a collection of path triangulations of a cell complex $K$ and its subcomplexes.
\qquad\textcolor{blue}{\Squaresteel}
\end{definition}

\begin{theorem}\label{thm:pathHsysRealization}{\bf Path Homotopy System Realization}\\
A path homotopy system has a geometric realization.
\end{theorem}
\begin{proof}
Let $h\mathop{\bigtriangleup}\limits^{\rho(K)} E$ be a path homotopy system and let $h\rho(K)\in E$ be a free group that presents either a path triangulation $h\bigtriangleup^n K$ of a complex $K$ or a subcomplex of $K$.  From Theorem~\ref{thm:pathTriangleRealization}, $h\bigtriangleup^n K$ has a geometric realization a collection of 1-cycles that are path-connected.  From Theorem~\ref{theorem:pathTriangulationGroup}, $h\bigtriangleup K$ has a free group presentation.  Then, $h\rho(K)$ has a geometric realization.  Hence, the homotopy system $h\mathop{\bigtriangleup}\limits^{\rho(K)} E$ has a geometric realization as collection of path triangulations on cell complex $K$ and its subcomplexes.
\end{proof}

\section{Path Triangulation as a Good Cover of a Cell Complex}
An important property of a path triangulation of a cell complex is that it provides a good cover of the complex.

Recall that a {\bf cover} of a space $X$ is a collection of subsets $E\in 2^X$ such that $X=\bigcup E$~\cite[\S 15.9,p. 104
]{Willard1970}. \\
 \vspace*{0.1cm}

\begin{definition}\label{def:goodCover}{\rm ~\cite[\S 4,p. 12]{Tanaka2021TiAgoodCover}}.\\
A cover of a space $X$ is a {\bf good cover}, provided, $X$ 
has a collection of subsets $E\in 2^X$ such that $X=\bigcup E$ and 
$\mathop{\bigcap}\limits_{\mbox{finite}} 
E\neq \emptyset$
is contractible,
 {\em i.e.}, all nonvoid intersections of the finitely many subsets $E\in 2^X$  are  contractible.
\qquad \textcolor{blue}{\Squaresteel}
\end{definition}
\vspace*{0.1cm}

\begin{theorem}\label{thm:pathTriangleGoodCover}{\bf Path Triangulation Good Cover}.\\
Every path triangulation of a cell complex is a good cover.
\end{theorem}
\begin{proof}
Here, we consider only good covers resulting from the trianglulation of finite planar cell complexes. Let $h\bigtriangleup^n K$ (briefly, $K$) be a path triangulation of a planar cell complex $K$. From Def.~\ref{def:triangulation},$K$ is a collection of path triangles $h\bigtriangleup E\in K$.  In addition, every vertex in $K$ is attached to one or more path triangles. Then we know that the intersection of all of the path triangles is a vertex $v\in K$, {\em i.e.}, we have
\begin{align*}
K &= \mathop{\bigcup}\limits_{h\bigtriangleup E\in h\bigtriangleup^n K}h\bigtriangleup E.\\
\mbox{nerve}\ \Nrv K &= \mathop{\bigcap}\limits_{h\bigtriangleup E\in h\bigtriangleup^n K}h\bigtriangleup E = \left\{p\right\}.
\end{align*} 
Hence, $\Nrv K$ contracts to vertex $p$ and $h\bigtriangleup E$ is a good cover.
\end{proof}

\section{Conjectures}
This section briefly introduces conjectures, which lead to a transition from the results in this paper to closely related results in digital topology by G. Lupton and N.A. Scoville~\cite{Lupton2019finitelyPresentedGroups}, {\em e.g.},

\begin{theorem}\label{thm:LuptonScoville2019}{\rm Theorem 5.16~\cite[p.22]{Lupton2019finitelyPresentedGroups}}.\\
Every finitely presented group occurs as the (digital) fundatmental group of some digital image.
\end{theorem}

\begin{conjecture}
Every path triangulation has a geometric realization in some  triangulated digital image.
\end{conjecture}

\begin{conjecture}
Every finitely presented path triangulation free group occurs as a free group presentation of some triangulated digital image.
\end{conjecture}

\section*{Acknowledgements}
Many thanks to Tane Vergili for her many brilliant observations and suggested improvements in various drafts of this paper.

\bibliographystyle{amsplain}
\bibliography{NSrefs}

\end{document}